\documentclass[12pt]{amsart}
\usepackage{amsthm,amsmath,amssymb}
\numberwithin{equation}{section}

\def\ca{{\mathcal A}}
\def\cb{{\mathcal B}}
\def\cc{{\mathcal C}}

\def\bc{{\mathbb C}}

\def\baf{{\mathbb F}}

\def\bn{{\mathbb N}}

\def\br{{\mathbb R}}

\def\bz{{\mathbb Z}}

\def\a{\alpha}

\def\b{\beta}

\def\j{\iota}
\def\g{\gamma}

\def\l{\lambda}

\def\f{\varphi}

\theoremstyle{plain}
\newtheorem{lemma}{Lemma}[section]
\newtheorem{proposition}[lemma]{Proposition}
\newtheorem{theorem}[lemma]{Theorem}
\newtheorem{corollary}[lemma]{Corollary}

\theoremstyle{definition}

\newtheorem{remark}[lemma]{Remark}
\newtheorem{definition}[lemma]{Definition}

\theoremstyle{remark}

\begin{document}


\begin{center}
Dedicated to the memory of Uffe V. Haagerup
\end{center}

\title[Rapid decay in  actions on a C*-algebra  ]{\textsc{  C*-dynamical rapid decay}}

\author[E.~Christensen]{Erik Christensen}
\address{\hskip-\parindent
Erik Christensen, Mathematics Institute, University of Copenhagen, Copenhagen, Demark.}
\email{echris@math.ku.dk}
\date{\today}

\begin{abstract}
Some well known results by Haagerup, Jolissaint and de la Harpe may be extended to the setting of a reduced crossed product of a C*-algebra $\ca$ by a  discrete group $G.$ We show that  for many  discrete groups, which include Gromov's hyperbolic groups and finitely generated discrete groups of polynomial growth, an inequality of the form 
$$\|X\| \leq C \sqrt{\sum_{g \in G}   (1+|g|)^4 \|X_g\|^2 } $$
holds for any finitely supported operator $X$ in the reduced crossed product.
\end{abstract}

\maketitle

\section{Introduction}
Any result in classical harmonic analysis naturally raises the question, does this extend to a non commutative setting ? In the situation of a discrete dynamical system, you may work with the group of integers $\bz$ which acts on a compact space as the group of homeomorphisms generated by a single homeomorphism. By Gelfand's fundamental theorem we know that the set of compact topological spaces correspond to unital commutative C*-algebras, and then the  classical discrete dynamical system as described above  can be made non commutative in 2 ways, either by studying a general non commutative C*-algebra equipped with a single *-automorphism or by  investigating properties of an action of a general discrete  group of homeomorphisms on a compact space. The construction named {\em reduced crossed product } of a C*-algebra by a discrete group encodes a set up for the study of the action of  a general discrete group on  a C*-algebra. 

Here we will extend results, by Haagerup  from \cite{Ha} on  free non abelian groups $\baf_d,$   by Jolissaint   \cite{Jo} on the concept named {\em rapid decay} and by de la Harpe \cite{PdH} on hyperbolic groups, from the setting of a {\em  reduced group C*-algebra } to the setting of a {\em reduced crossed product C*-algebra. }   The C*-algebra $\ca $ upon which the discrete group $G$ acts by *-automorphisms $\a_g$ may be abelian or non abelian, and the case where the algebra is the trivial $\ca = \bc I$ is the reduced group C*-algebra case. 

Uffe Haagerup's article \cite{Ha}  has been very influential in the study of discrete groups and von Neumann algebras of type II$_1,$ and his name and  results are linked to many fundamental mathematical properties such as the {\em Haagerup approximation property } \cite{CCJJV}, \cite{HK},  the {\em completely bounded approximation property } \cite{CH}  and the {\em Haagerup tensor product } \cite{Pa}.  We will not go into the study of any of these aspects but concentrate on an extension of the first basic result from \cite{Ha} and lift it from the reduced group case to the reduced crossed product case for discrete hyperbolic groups acting on general C*-algebras. In this setting we prove that there exists an  $C>0,$ depending on the group only, such that for any  linear combination $X = \sum L_gX_g$ in the algebraic reduced crossed product we have the following inequality
\begin{equation}  \label{Hag1}  \|X\| \leq C \sqrt{\sum_{g \in G }(1+ |g|)^4\|X_g\|^2}\,
\end{equation} 
which is a direct generalization of Haagerup's estimate in the group algebra case, since for these groups $C< 2.$  
In the article \cite{Jo} Jolissaint showed that, in the group algebra case,  inequalities like (\ref{Hag1}) may be obtained for many other discrete groups, which are not free, and he introduced the property named {\em rapid decay } for a discrete group $G $ with a length function $|g|,$ if  there exist positive constants $C, r$ such that for any $x = \sum x_g\l_g$ in $\bc[G] $ 

\begin{equation}  \label{Jo1}  \|x\|_{op} \leq C  \sqrt{\sum (1+|g|)^{2r}|x_g|^2} .
\end{equation}

Shortly after Jolissaint had obtained his results they were extended to the setting of hyperbolic groups by de la Harpe \cite{PdH}, and it turns out that  it is possible to extend both Jolissaint's and de la Harpe's results to the crossed product setting. The basic insight which makes this possible was formulated by Jollisaint in his lemmas 3.2.2 and 3.2.3. Then de la Harpe proved that these lemmas are true in the setting of discrete finitely generated hyperbolic groups, so Jolissaint's findings could be extended. Here we have collected the statements from Jolissaint's basic lemmas into a property (J) which a discrete group with a length function may have, and then we prove that the inequality (\ref{Hag1}) holds,   in any reduced crossed product of a C*-algebra and a group satisfying property (J). It seems possible to us that property (J) implies hyperbolicity, but we have very little experience in dealing with such a question. 

We have also considered reduced crossed products of C*-algebras by  discrete finitely generated groups which have polynomial growth. In this case it is possible to get the following - much better - result. If a discrete finitely generated group has polynomial growth then there exist $C> 0,$ $s>0$ such that for any finitely supported operator $X = \sum L_gX_g$ in a reduced crossed product $C^*_r(\ca\rtimes_\a G)$ we have 
\begin{align} \label{Pgrow}  
\|X\| \leq C& \|  \sum_{g \in G} (1 +|g|)^{2s} L_gX_gX_g^*L_g^*\|^{\frac{1}{2}} \\ \notag \|X\| \leq C&  \|  \sum_{g \in G} (1 +|g|)^{2s} X_g^*X_g\|^{\frac{1}{2}}
\end{align}

The advantage of (\ref{Pgrow}) over (\ref{Jo1}) is that usually
 $$ \|\sum_{g \in C_k} X_g^*X_g\|^{\frac{1}{2}} < \big(\sum_{g \in C_k} \|X_g^*X_g\|\big)^{\frac{1}{2}}.$$
 
\section{ Notation and norms}
Most of the content of this section is well known in the setting of a reduced group C*-algebra,  but here we deal with a reduced crossed product C*-algebra. In order to be able to generalize the methods from Haagerup's paper \cite{Ha}, we need a characterization of the operator norm in a reduced crossed product of a C*-algebra by a discrete group, and this result, which we think might be known but unpublished,  is presented in Proposition \ref{cpnorm} in a self contained frame.    We start with a well known  operator theoretical version of Cauchy-Schwarz' inequality, \cite{Pa}.

\begin{definition} \label{crNorm}
Let $H,\,  K $ be  Hilbert spaces, $J$ be an index set and 
$(a_\j)_{(\j \in J)}$ a family of bounded operators in $B(H,K).$
\begin{itemize}
\item[(i)] If the sum $\sum_\j a_\j^*a_\j $ is ultrastrongly convergent  in    $B(H)$ we say that the family $(a_\j)$ is {\em column bounded} with column norm $\|(a_\j)\|_c := \|\sum a_\j^* a_\j\|^{\frac{1}{2}}.$ 
\item[(ii)] If the sum $\sum_\j a_\j a_\j^* $ is ultrastrongly convergent    in $B(K)$ we say that the family $(a_\j)$ is {\em row bounded} with row norm $\|(a_\j)\|_r := \|\sum a_\j a_\j^*\|^{\frac{1}{2}}.$ 
\end{itemize} 
\end{definition}
\begin{proposition} \label{csOpIneq}
Let $H,\,  K $ be  Hilbert spaces, $J$ be an index set and 
$(a_\j)_{(\j \in J)},$ $(b_\j)_{(\j \in J)}$ be column bounded  families of operators in $B(H,K).$ 
\begin{itemize}
\item[(i)] The sum $\sum_\j a_\j^* b_\j $ is ultrastrongly convergent in $ B(H)$  and the operator norm of the sum satisfies $\|\sum_\j a_\j^* b_\j \| \leq \|(a_\j)\|_c \|(b_\j)\|_c.$
\item[(ii)] There exists a column bounded family $(e_\j)_{(\j \in J)}$ of column norm at most $1$  such that for the positive  bounded operator $h$ on $H$ defined by $h := (\sum_\j a_\j^* a_\j)^{\frac{1}{2}}$ we have for each $\j  \in J,  \, a_\j = e_\j h $ and the sum $\sum_\j e_\j^*e_\j$ equals the range projection of $h.$
\item[(iii)] Let $k := (\sum_\j b_\j^* b_\j)^{\frac{1}{2}},$ then there exists a contraction $c$ in $B(H)$ such that $\sum_\j b_\j^* a_\j = kch.$
\end{itemize} 
\end{proposition}

\begin{proof}
The families $ (a_\j)$ and $(b_\j)$ represent bounded column operators in the operator space $M_{(J, 1)}\big(B(H,K)\big)$ and the statement (i) follows from properties of the operator product.

 The statement (ii) follows from the polar decomposition applied to the column operator $(a_\j).$

 The result in statement (ii) may be applied to the column operator $(b_\j) $ such that each 
$b_\j  = f_\j k, $ we can then define a contraction $c$ in $B(H) $ by $c:= \sum_\j f_\j^*e_\j$ and statement (iii) follows. 
\end{proof}

The rest of this article takes place in the setting of  the reduced  crossed product of a C*-algebra $\ca$ by a discrete group $G,$ which acts on $\ca $ by the *-automorphisms $\a_g.$ We made a study of the properties of this crossed product  in the article \cite{Ch} and we will use most  of the notation and several of the  results of that article below.  A basic point of view in Section 2 of \cite{Ch} is  that there are many facts related to properties of the coefficients of a Fourier series which generalize to properties of the coefficients of an element in a reduced discrete C*-crossed product. 

We recall that any  element $X$ in $  \cc := C^*_r(\ca \rtimes_\a G)$ has a Fourier series expansion $X \sim \sum_{g \in G} L_g X_g,$ where the sum is convergent in the norm $\|.\|_\pi$ described in Proposition 2.2 of \cite{Ch}. A simple computation shows that for $X \sim \sum L_gX_g$ we have that the column and row norms of the family $(L_gX_g)_{(g \in G)}$ may be computed as \begin{align}
\|(L_gX_g)_{(g\in G)}\|_c^2& = \|\pi(X^*X) \| = \|(X^*X)_e\|  \\ \notag \|(L_gX_g)_{(g\in G)}\|_r^2 &= \|\pi(XX^*) \| = \|(XX^*)_e\|. 
\end{align}   
In particular we notice the following proposition. \begin{proposition}
Let $X \sim \sum_g L_g X_g $ be an element in $\cc$ then the sum converges in the column norm. 
\end{proposition}

We will use the $\pi-$norm or column norm to estimate the operator norm in the computations to come, so we need the following proposition. It  may be known to several people, but may be in a slightly different setting. We are not aware of an explicit formulation as the one we present, but the results of \cite{PSS} have a similar flavour. We do also think that people who prefer to look at completely positive mappings as correspondences or  operator bimodules may know the result, but still we are missing a reference. 

\begin{proposition} \label{cpnorm}
Let $\cb$ be a C*-algebra, $H$ a Hilbert space and $\pi :\cb \to B(H)$ a completely positive and faithful mapping then
\begin{equation}
\notag \forall b \in \cb: \|b\| = \sqrt{ \sup\{\|\pi(x^*b^*bx)\|\, : \, \|\pi(x^*x)\| \leq 1\}}.
\end{equation}
\end{proposition}
\begin{proof}
We may suppose that $\cb $ is a subalgebra of $B(K)$ for some Hilbert space $K,$ then since $\pi$ is completely positive and faithful there exists  by Stinespring's result \cite{St}  a faithful representation $\rho$ of $\cb$ on $H$ and a bounded operator $C$ in $B(K, H)$ such that $$ \forall b \in \cb: \quad \pi(b) = C^*\rho(b)C.$$ We may define a semi norm, say $n,$ on  $\cb $ by 
\begin{equation} \label{n(b)}  \forall b \in \cb: \quad n(b) := \sup\{ \|\rho(bx)C\| : \|\rho(x)C\| \leq 1 \}.
\end{equation}
Since $\|C^*\rho(y^*y)C\| = \|\rho(y)C\|^2 $ for any $y$ in $\cb$ and $\pi$ is faithful we get that $n$ is a norm and that
 \begin{equation}
\notag \forall b \in \cb: n(b) = \sqrt{ \sup\{\|\pi(x^*b^*bx)\|\, : \, \|\pi(x^*x)\| \leq 1\}}.
\end{equation}
 On the other hand the definition (\ref{n(b)}) implies that for $b, d$ in $\cb$ we have $n(bd) \leq n(b)n(d)$ so $n$ is an algebra norm and  $n(b) \leq \|b\|.$
 For any pair $b, x $ in $\cb$ with $\|\rho(x)C\|^2 =\|\pi(x^*x)\| \leq  1$ we have 
 \begin{align*} n(b^*b) &\geq \|\rho(b^*bx)C\| \geq \|C^*\rho(x^*b^*bx)C\|  = \|\rho(bx)C\|^2 , \text{ so } \\ n(b^*) n(b) & \geq n(b^*b)  \geq n(b)^2. \end{align*}
From here it follows that $n(b) = n(b^*)$ and then $n(b^*b) = n(b)^2,$ and the completion say $\hat \cb$  of $\cb$ with respect to the C*-norm $n$ becomes a C*-algebra such that  the  inclusion of $\cb $ in $ \hat \cb$ is a contractive faithful *-homomorphism and hence an isometry. The proposition follows. 
\end{proof}

In Haagerup's  and Jolissaint's articles, \cite{Ha}, \cite{Jo} they use the symbol $*$ to denote the operator product in a group algebra, since this is really a convolution product. In  our article on crossed product C*-algebras \cite{Ch} the operator product is an invisible dot and the $*$ is used to denote the Hadamard product, which in the notation from above takes the form $$\big(\sum_g L_gX_g\big)*\big(\sum_h L_hY_h\big) := \sum_f L_fX_fY_f.$$  We will use this convention here, too.

\section{ The property (J)} \label{PJ}

The basic ideas in the arguments to come are taken from Haagerup's article, and in the setting of a non commutative free group it is clear that for two group elements $x, y$ with reduced words  $x = x_1 \dots x_k$ and $y = y_1 \dots y_l$ the number of cancellations, say $p,$ needed to spell to $xy$ gives the spelling of $xy$ directly as $xy = x_1 \dots x_{(k- p)} y_{(p+1)} \dots y_l.$ In Jolissaint's article he uses this observation in a very clever original way, and he shows that this idea may be generalized to work {\em up to a controllable error }  in some groups of isometries on a Riemannian manifold with bounded strictly negative sectional curvature. Then de la Harpe showed that Jolissaint's method of dealing with cancellations works in any finitely generated discrete group which is hyperbolic, as defined by Gromov, \cite{Gr} and \cite{GH}. Here we will instead take these results as the basis for the definition of a property we have named (J).   

We will now define the setting in which we will use Haagerup's, Jolissaint's and de la Harpe's ideas. We define the cancellation number in a general group with a length function as follows. 
\begin{definition}
Let $G$ be a group with a length function $g \to |g|.$ For  $g,h$  in $G$ the cancellation number $c(g,h)$ of the pair is defined as the non negative integer $p(g,h) $ which satisfies  $$2p(g,h) \leq  |g| + |h| - |gh|< 2p(g,h) +2.$$
\end{definition} 
It follows from the properties of a length function that $ 0 \leq p(g,h)  \leq \min\{ |g|, |h|\},$  and the cancellation number divides the cartesian product $ G \times G$ into a sequence of disjoint subsets $(P_p)_{( p \in \bn_0)} $  defined by  $P_p:= \{(g,h) \in G \times G\, : \, p(g,h) = p.\}.$ 

Following Jolissaint we define certain subsets of a group $G$ with a length function $|g|$ as follows
\begin{definition}
\begin{itemize}
\item[]
\item[(i)] $ \forall r  \geq 0: 
\quad \quad \quad \quad \, \, \,\, \, B_r := \{g \in G\,:\, |g| \leq r\},$
\item[(ii)] $ \forall k  \geq 0: \,
\quad \quad \quad \, \,  \, \, \quad C_k := \{g \in G\,:\, k-1 < |g| \leq k\},$
 \item[(iii)] $ \forall   k \geq 0 \, \forall \a \geq 0 :  \quad C_{k, \a} := \{g \in G\,:\, k - \a \leq |g| \leq k + \a \},$ 
\end{itemize}
\end{definition}

We can now collect  the sufficient conditions, we have have dragged out of \cite{Jo}, into a property we name (J). We will not focus on which groups that may satisfy the property (J), but we think that the survey article on rapid decay \cite{CI} by Chatterji will show that many groups do have property (J).   On the other hand we will sketch arguments which show that Jolissaint's and de la Harpe's examples do have property (J). It is easy to see that free non abelian groups do have property (J) with the extra  very nice properties that the constants of the definition satisfy    $\a=\b= \g= 0$ and $N=1.$  
 
\begin{definition} \label{RdDef}
Let $G$ be a discrete group with a length function $g \to |g|.$ The pair $(G, |\cdot|)$ has   property, (J) if
\begin{align} &\exists \a >0, \b>0, \g>0,  : \\ 
(i)\notag  \,  &  \forall g \in G, \forall 0 \leq s < |g| +1, \,  \exists u_{(g,s)}  \in C_{s, \a}\\ \notag (ii)  \,&  v_{(g,s)} := u_{(g,s)}^{-1} g \in C_{(|g|+1 -s), \b}  \\ (iii) \, & \notag
\forall p \in \bn_0,  \forall  (a, b) \in P_p\,     \exists\,  c(a,b)  \in C_{p, \g} \text{ s. t. } \\ \,  & \notag  a = u_{(ab,(|a|-p))} c(a,b), \, b = c(a,b)^{-1} v_{(ab, (|a|- p))}.\\ \label{Nsolut}
 & \forall \mu >0,  \nu >0 \, \exists N \in \bn: \\ \notag  &  \forall b \in G \forall \, 0 \leq p \leq |b| \,:\, \,  |\{ (c, v) \in C_{p, \mu} \times C_{(|b| - p), \nu} :  c^{-1}v = b\}| \leq N.
\end{align}
\end{definition}

\begin{remark}
It is important for the following proofs in the next section to notice that the group element $u_{(g,s)} $   is  determined uniquely by $g$ and $s,$ and then $v_{(g,s)} = u_{(g,s)}^{-1}g$ is also determined by $g$ and $s$ only. 

The interesting thing about the  factors  $c(a,b)$ is that they always approximately satisfy $|c(a,b)| = p.$  
\end{remark}

In Jolissaint's  proof the group element $u_{(g,s)} $ is chosen geometrically, as we sketch now.  On the geodesic which connects a point $m$ in the manifold with its image $g(m),$ one  chooses the point $n_s$  which has the distance $s$ to $g(m) . $ Then $u_{(g,s)}$  is chosen such   that the distance between $n_s$ and $u_{(g,s)}^{-1} \big(g(m)\big) $ is minimal among the distances between $n_s$ and the set $\{u(m)\,:
\, u \in G\}.$ We will not continue to quote Jolissaints proof, but show that  a group $G,$ which is hyperbolic with respect to a given word length, has the property (J). The proof follows that of de la Harpe \cite{PdH}, but in order to explicitly establish the property (J) from the definition above, we repeat part of it here. We will use the following notation. For a real number $s,$ the expression $\lfloor s \rfloor $ means the largest integer dominated by $s.$

\begin{lemma}
Let $G$ be a finitely generated discrete group which is hyperbolic with respect to  word length. 
Then $G$ has the property (J).
\end{lemma} 

\begin{proof}
For a group element $g$ written in reduced form as $g = g_1 \dots g_{|g|}$ and a real $s, \, 0 \leq s < |g| +1$ we define $$u_{(g,s)} := \begin{cases} e \quad \quad \quad \,\, \,\text{ if } 0  \leq s < 1 \\ 
g_1 \dots g_{\lfloor s \rfloor } \text{ if } 1 \leq s < |g|+1, \end{cases} $$ and we find that $u_{(g,s)} \in C_{s,1} $ so $\a =1$ may be used.
Similarly we find $$ v_{(g,s)} = \begin{cases}
 g_{ ( \lfloor s \rfloor + 1)} \dots g_{|g|}    \text{ if } \, \,0 \leq s < |g| 
\\ e \quad \quad \quad \,\, \quad \quad \text{ if } |g| \leq s < g+1, \end{cases}$$ and we get $v_{(g,s)} \in C_{(|g| +1 -s),1},$ so $\b = 1 $ is possible.

Given a non negative integer  $p$ and a pair of group elements $(a,b) \in P_p$ with $ab = g,$ then for $k:= |a|, $ $l:= |b|$ there exists $c \in \{0, 1\}$ such that $|g| = k+l-2p -c.$ Then for $u_{(g,k-p)} $ we may apply the  lemma at the bottom of page 771 in \cite{PdH}, to see that there exists an $M \geq 0,$ independent of $k, l, p,  c$ such that for $c(a,b) := u_{(g, k-p)}^{-1}a$ the following  inequalities hold \begin{equation}
p \leq |c(a,b)| \leq p + M,
\end{equation}
  so $c(a,b) \in C_{p , M }.$  

In order to establish the property (\ref{Nsolut}) we remark, that in our case we have $ v_{(g,(k-p))}  \in C_{(|g| + 1 - k +p),1} = C_{(l-p -c +1), 1},$ which means $$ c(a,b) \in C_{p, M}, \, \, v_{g, (k-p)} \in C_{(|b|-p), 2} \text{ and } c(a,b)v_{(g, (k-p)) } = b. $$ The result then follows from item (ii) in the lemma of \cite{PdH}.

\end{proof}

\section{Rapid decay} 
The most basic example of the phenomenon named {\em rapid decay } by Paul Jolissaint in \cite{Jo} is presented quite early in many courses on Fourier series. The example tells, that if  $f(t)$ is a differentiable complex \newline $2\pi-$periodic function on $\br,$ then its Fourier series is uniformly convergent.  This is proven via the following argument based on the Cauchy-Schwarz inequality, as follows.   
 Let $f(t)$ have the Fourier series $f(t) \sim \sum_\bz c_ne^{int},$ then the derivative $f^\prime (t)$ has the Fourier series   $f^\prime(t) \sim  \\ \sum_\bz in  c_n e^{int},$ and the sequence of complex numbers $(nc_n)_{(n \in \bz)}$ is in $\ell^2(\bz).$  Since for $n \neq 0$ we may write $c_n = \frac{1}{n}(nc_n),$ we get that  the sequence  $(c_n)_{(n \in \bz)} $ is in $\ell^1(\bz)$ with $\|(c_n)\|_1 \leq  |c_0| + \frac{\pi}{\sqrt{3}}\|(nc_n)\|_2.$ If we translate this to the setting of the discrete group $\bz$ equipped with the natural length
  function  $|n|, $ we find that the group algebra $C^*_r(\bz)$ may be identified with the complex  continuous $2\pi-$periodic functions on the real axis and an  element $x$  in the group algebra which correspond to a  differentiable function has a presentation as a uniformly convergent sum $x = \sum_{(n \in \bz)} x_n\l_n.$ This example may be generalized to the setting of a discrete group with a length function when the content of the example is formulated as in the following proposition. 
   \begin{proposition} Let $x \sim \sum_\bz x_n\l_n $ be an operator in $C^*_r(\bz).$ If the sequence $\big((1+ |n|)x_n\big)_{( n \in \bz)} $ is in $\ell^2(\bz),$ then the series is uniformly convergent and \newline $\|x\| \leq \big(\frac{\pi^2}{3}-1\big)^{\frac{1}{2}}\big( \sum_{n \in \bz} |x_n|^2(1 + |n|)^2\big)^{\frac{1}{2}}. $\end{proposition}
   
We recall from Chatterji's survey article  \cite{CI},   Definition 2.9, in a modified form.
   
  \begin{definition}
Let $G$ be a discrete group with a length function $|g|,$ then  $G$ has the {\em rapid decay property, }with respect to $|g| $  if there exists positive constants $C, s$  such that for any operator $x = \sum_g x_g \l_g$  in the reduced group C*-algebra and with finite support  we have $$\|x\| \leq C\sqrt{\sum_{g \in G}|x_g|^2(1 +|g| )^{2s}}.  $$  
  \end{definition}
  
    It is immediate that this definition may be 
   extended to the setting of a reduced crossed product of a C*-algebra by a discrete group in several ways.
   We have played with  3 possibilities of definition, but only been able to obtain results for the 2 of them, which we define below. The third possibility is mentioned after the definition.    
  \begin{definition}
Let $G$ be a discrete group with a length function $g \to |g|,$ such that $ G$ acts on a C*-algebra $\ca$ via a group of *-automorphisms $\a_g.$  
The reduced crossed product $\cc:= C^*_r( \ca \rtimes_\a G)$  has  {\em rapid decay } of {\em operator type  } or {\em scalar type} if 
there exist positive constants $C, s$ such that for any operator $X = \sum L_gX_g$ with finite support:
\begin{align*} 
&\text{ operator type:}\\ &\|X\| \leq C\| \sum_{g \in G} (1 + |g|)^{2s} (L_gX_gX_g^*L_g^*+ X_g^*X_g)\|^{\frac{1}{2}}\\
&\text{ scalar type:}\\  &\|X\| \leq C\bigg( \sum_{g \in G} (1 + |g|)^{2s} \|X_g\|^2\bigg)^{\frac{1}{2}}
\end{align*}
If one of the properties above holds for any C*-algebra $\ca$ carrying an action $\a_g$ of 
$G$ we say that $G$ possesses {\em  complete rapid decay } of respectively  operator type and  scalar type. 
\end{definition}  

In the very first version of the article we thought that we could prove that free non abelian groups do have a sort of mixed rapid decay defined as 

\begin{align*} 
&\text{ mixed type:}\\ &\|X\| \leq C\bigg( \sum_{k=0}^\infty(1 + |k|)^{2s} \|\sum_{g \in C_k} \big(L_gX_gX_g^*L_g^*+ X_g^*X_g\big)\|\bigg)^{\frac{1}{2}}.
\end{align*} 

 Unfortunately we were wrong, but it may still be that some discrete groups with non polynomial growth satisfy such a condition, and that would be very helpful in the study of multipliers of the form $M_\f$, as it follows from the proof of Proposition \ref{Mult}.

The {\em complete operator type of rapid decay } may be established for  finitely generated groups with polynomial growth by a simple imitation of the proofs from the group algebra case. For other groups it seems impossible to us to establish the operator type of rapid decay outside the reduced group algebra case. 
We establish the complete {\em scalar  type of rapid decay } for discrete groups which possess  property (J) with respect to a length function.

  \section{complete operator rapid decay in a reduced crossed product of a C*-algebra by a finitely generated group with polynomial growth}
  
Here we modify some of Jolissaint's results from his section 3.1 of \cite{Jo}  to cover our situation.    
  
  \begin{theorem}
  Let $\ca$ be a C*-algebra, $G$ a finitely generated discrete group with polynomial growth which has an action $\a_g$ on $\ca$  as a group of *-automorphisms. There exists positive reals $M, s$ such that for any  finitely supported operator $X= \sum L_gX_g $ in $\cc:= C^*_r(\ca\rtimes_\a G):$
   \begin{align}\notag
   \|X\|& \leq M \|  \sum_{g \in G}(1+|g|)^{s+2}L_gX_gX^*_gL^*_g\|^{\frac{1}{2}}\\ 
   \notag\|X\| & \leq M \|  \sum_{g \in G}(1+|g|)^{s+2}X^*_gX_g\|^{\frac{1}{2}}.
   \end{align}
  \end{theorem}
\begin{proof}
The polynomial growth implies that there exists positive reals $C, s$ such that $|C_k | \leq C(1+k)^s, $ then the following manipulations are standard, and the proof follows from  Proposition \ref{csOpIneq} as follows 
 
\begin{align*}
X = &\sum_{k=0}^\infty \sum_{g \in C_k}  \frac{1}{(1+k)|C_k|^{\frac{1}{2}}}\big((1+k)|C_k|^{\frac{1}{2}}L_gX_g\big)\\ 
\leq &\sqrt{C} \sum_{k=0}^\infty \sum_{g \in C_k}  \frac{1}{(1+k)|C_k|^{\frac{1}{2}}}\big((1+k)^{(1 + \frac{s}{2})}{\frac{1}{2}}L_gX_g\big)\\ 
\|X\| \leq & \frac{\pi\sqrt{C}}{\sqrt{6}}
\|\sum_{k =0}^\infty \sum_{g \in C_k} (1+k)^{(2 +s) } X_g^*X_g \|^{\frac{1}{2}}\\
\leq &\frac{\pi\sqrt{C}}{\sqrt{6}}\sqrt{2} 
\|\sum_{g \in G} (1+|g|)^{(2 +s) } X_g^*X_g \|^{\frac{1}{2}}
\end{align*} so $M :=\frac{\pi\sqrt{2C}}{\sqrt{6}},$ may be used. The inequality involving $L_gX_gX_g^*L_g^*$ follows in the same way.
\end{proof}

\section{Complete scalar rapid decay for discrete groups with property (J) }

The basic result in this section is the proposition just below, and this is a combined extension of some of the first results in \cite{Ha}.
\begin{proposition}
Let $(G, |\cdot|) $ be a discrete group with a length function satisfying the property  (J).   There exists a positive constant $N$ such that for any action of $G$  as a group of *-automorphisms $\a_g$ on a C*-algebra $\ca,$ any non negative integer $k$ and  any  operator $X = \sum_{g \in G}L_gX_g$ in $\cc := C^*_r(\ca \rtimes_\a G)$  with finite support in $C_k:$
$$\|X\| \leq N(1+k)\sqrt{\sum \|X_g\|^2 }.  $$
\end{proposition}
\begin{proof} 

We will let $Y$ denote any finitely supported  element in $\cc$ with column  or $\pi-$norm at most 1. By Proposition \ref{cpnorm} it is sufficient to bound the $\pi-$norm of $XY$ in order to bound the operator norm  $\|X\|.$  Using the cancellation numbers, the operator $XY$ may be written as a sum of $k+1$ summands $S_p$ defined as follows  
\begin{align} \label{JoSp}
\notag XY &= \sum_{a \in C_k} \sum_{b \in G} L_aX_aL_bY_b  \\
&= \sum_{p = 0}^k \sum_{\{(a,b) \in P_p \, : a \in C_k\}}L_aX_aL_bY_b \\ \notag &= \sum_{p=0}^k S_p 
\end{align}
 
Then let us fix a $p$ in the set $\{0, 1, \dots, k\}.$ By the property (J) there exists $\a >0, \hat \b>0, M >0$ such that to  each group element $g,$  each  $p$ and each pair of group elements $(a, b) \in P_p$ with $ab = g$ and $a \in C_k$ there exist group elements $u_{(g, (k-p))} ,\, v_{(g, (k-p))}, \, c(a,b) $ such that 
\begin{align}
&u_{(g,(k-p))} \in C_{(k-p), \a},\, \quad v_{(g,(k-p))} \in C_{(|g|+1 -k +p), \hat \b }, \quad c(a,b)  \in C_{p, M},\\ \notag \,& a= u_{(g,(k-p))}c(a,b) ,\, \quad b= c(a,b)^{-1}v_{(g,(k-p))} .
\end{align}
We have $|g| = |a|+|b| - 2p - c$ with $c \in \{0,1\},$ hence since $|a| = k $ we get $|g|+ 1 -k +p  = |b| -p +1 -c$ so for $\b:= \hat \b + 1 $ we have  \begin{equation} \label{Hag} v_{(g,(k-p))} \in C_{(|b| - p), \b}. 
\end{equation}  
In Haagerup's case with free groups we get $ u_{(g, (k-p))}  = a_1 \dots a_{(k-p)} ,\newline \, c(a,b)  = a_{(k-p+1)} \dots a _k, \, v_{(g, (k-p))}  = b_{(p+1)}\dots b_{|b|},$ with \newline $u_{(g, (k-p))} \in C_{(k-p)}, \, c(a,b)  \in C_p, \, v_{(g, (k-p))} \in C_{(|b|- k)}.$  In the case of a free non commutative group  the 2 first elements $u, c$ are determined by $a, p$ and the last 2 elements $c, v$ are determined by $b, p,$ but in the general case $a,p$ does not determine the pair $u_{(g, (k-p))} , c(a,b)$ nor does the pair $b,p$ determines the pair $c(a,b), v_{(g,(k-p))}.$ The condition (\ref{Nsolut}) is designed to deal with this problem.

We can now start the estimation, and we will use the results of Proposition \ref{csOpIneq}, so for a given $g$ we define a positive operator $Q_{u_{(g, (k-p))}}$ by \begin{align} \label{Q}
Q_{u_{(g, (k-p))}}^2 :&= \sum_{\{ (a,b) \in P_p: a \in C_k, \, ab =g \}}L_aX_aX_a^*L_a^*  . 
\end{align} 
To each pair $(a,b)$ in $P_p$ with $a$  in $C_k,$ and $ab= g$ there exists a contraction $q(a,b) $ such that $L_aX_a = Q_{u_{(g, (k-p))}} q(a,b)$ with  $$\sum_{\{(a,b ) \in P_p: a \in C_k, ab =g\}}q(a,b)q(a,b )^* \leq I.$$
Analogously we define $R_{v_{(g, (k-p))}}$ as the positive operator which is given by   
\begin{align} \label{GG} 
R_{v_{(g, (k-p))}}^2 &=
\sum_{\{(a,b) \in P_p : a \in C_k \, ab = g \}} Y_b^*Y_b 
\end{align}
To each group element $g$ and each pair $(a,b)$ in $P_p$ with $g = ab$ and $a$  in $C_k,$ there exists a contraction $r(a,b) $ such that $L_bY_b = r(a,b)R_{v_{(g, (k-p))}}$ with  $$\sum_{\{(a,b ) \in P_p: a \in C_k, ab =g\}}r(a,b)^*r(a,b ) \leq I,$$ and according to Proposition \ref{csOpIneq} we may define a contraction operator $m_g $ by 
\begin{equation}
\forall g \in G: \quad m_g := \sum_{\{(a,b) \in P_p : a \in C_k \, ab = g \}} q(a,b)r(a,b). 
\end{equation}
When combining these equations we get
\begin{equation}
S_p(g) = Q_{u_{(g, (k-p))}}m_gR_{v_{(g, (k-p))}}
\end{equation}
and then 
\begin{align} \label{sumSp}
& \sum_{g \in G} S_p(g)^*S_p(g) = \sum_{g \in G } R_{v_{(g, (k-p))}}m_g^*Q_{u_{(g, (k-p))}}^2m_gR_{v_{(g, (k-p))}} 
\text{by }  \|m_g \| \leq 1 \\  \notag \leq  & \sum_{g \in G} \|Q_{u_{(g, (k-p))}}^2\| R_{v_{(g, (k-p))}}^2 \text{ by }  (\ref{Q}) \text{ and } (\ref{GG}) \\ 
\notag  \leq & \sum_{g \in G} \big(\sum_{(a,b) \in P_p: ab =g }\|X_{u_{(g,(k-p))}c(a,b)}\|^2\big)
\cdot \\ \notag &
 \big(\sum_{(e,f) \in P_p: ef =g }Y_{c(e,f)^{-1}v_{(g,(k-p))}}^*Y_{c(e,f)^{-1}v_{(g,(k-p))}}\big)\text{ split  to sum over } h, \, g  \\ \notag \leq & \big( \sum_{g \in G} \sum_{(a,b) \in P_p: ab =g }\|X_{u_{(g,(k-p))}c(a,b)}\|^2\big)\cdot \\ \notag & \big(\sum_{h\in G} \sum_{(e,f) \in P_p: ef = h}Y_{c(e,f)^{-1}v_{(g,(k-p))}}^*Y_{c(e,f)^{-1}v_{(g,(k-p))}} \big)
\end{align} 

In the last 2 sums depending on $g$ and $h$ respectively, one can via the property (\ref{Nsolut}) get an upper bound on the number of times each element  of the form $\|X_a\|^2$ or $Y_f^*Y_f$ appears in the sum. Let $a$ in $C_k$ be given, then the  number of solutions to the equation  \begin{equation} \notag
u \in C_{(k-p), \a },\, c \in C_{p, M} : \quad uc = a
\end{equation} is at most $N.$ Similarly for each $f$  the number of solutions to the equation 
\begin{equation} \notag
v \in C_{(|f|-p), \b },\, c \in C_{p, M} : \quad c^{-1} v  = f
\end{equation} 
is at most $N,$ and hence by  (\ref{sumSp}) 

\begin{align} \label{normX}
\|S_p\|_\pi \leq & N\big(\sum_{a \in C_k}  \|X_a\|^2\big)^{\frac{1}{2}} \|Y\|_\pi \text{ so } \\ \notag \|X Y \|_\pi \leq & (k+1) N\big(\sum_{a \in C_k}  \|X_a\|^2\big)^{\frac{1}{2} }\|Y\|_\pi  \text{ by Proposition }  \ref{cpnorm}\\ \notag \|X\| \leq &(k+1) N\big(\sum_{a \in C_k}  \|X_a\|^2\big)^{\frac{1}{2}},
\end{align}
and the proposition follows 

\end{proof} 

In the case of a free non abelian group $\baf_d$ with  any set, finite or infinite,  of generators the proposition above holds with $N= 1.$ 
The reason is that for a pair $(a, b) $ in $P_p$ with $a \in C_k$  the decompositions $ a = u_{(g , (k-p))} c(a,b ) $ and $b  = c(a,b)^{-1}v_{(g, (k-p))} $
 are described in an exact form just below the equation (\ref{Hag}), such that the constants  $\a, \,\b, \,  \g$ in Definition \ref{RdDef} may be used with value 0. 
 The exact decomposition also shows that in this case there will only be one solution to the equation (\ref{Nsolut}) so we get $N=1$ in this case, and we may note the following corollary. 

\begin{corollary}
Let $\ca$ be a C*-algebra with an action $\a_g$ of $\baf_d,$ the free non commutative group with $d$ generators, then for any $k \in \bn_0$ and  any $X$ in $C^*_r(\ca)$ with finite  support in $C_k $:
$$ \|X\| \leq (k+1)\big( \sum_{a \in C_k} \|X_a\|^2\big)^{\frac{1}{2}}.$$ 
\end{corollary} 

It is worth to remark that the result in the corollary above in the case of the trivial C*-algebra $\ca = \bc I$ gives exactly the content of Lemma 1.5 in \cite{Ha}.  The content of Lemma 1.3 in \cite{Ha} is named the { \em Haagerup property  } in \cite{OR}.  
We may also here add a corollary describing the form the {\em Haagerup property}  takes in the setting of a reduced crossed product of a C*-algebra by a discrete group with the property (J)
This is the natural extension of Jolissaint's Proposition 3.2.4 in \cite{Jo}. 

\begin{corollary} \label{HagProp}
For $k,l,m $  non negative integers and any pair of finitely supported elements $X= \sum_g L_gX_g$ with support in $C_k$  and $Y = \sum_g L_gY_g$  with support in $C_l:$  

\begin{align*} & \|M_{\chi_m}*\big(XY\big) \|  \leq  N\big(\sum_{a\in C_k}\|X_a\|^2 \big)^{\frac{1}{2}}\|\sum_{b\in C_l}Y_b^*Y_b \|^{\frac{1}{2}}\\ 
& \|M_{\chi_m}*\big(XY\big) \|  \leq  N\|\sum_{a\in C_k}L_aX_aX_a^*L_a^*\|^{\frac{1}{2}}\big(\sum_{b\in C_l}\|Y_b\|^2\big)^{\frac{1}{2}}
\end{align*}

When $G$ is a free non abelian group $N=1.$ 
\end{corollary}

\begin{proof}
Let $a \in C_k$ and $b \in C_l$ be such that $ab \in C_m$ then there exists uniquely determined  $c \in \{0, 1\} $ and $p$ in $\bn_0$ such that $$ k+l -m = 2p+c . $$ In particular at most one $S_p \neq 0, $  and the first inequality of the corollary follows from the proof of the theorem. The second follows from the first when  applied to $(Y^*X^*).$ The free group statement follows from the corollary just above. \end{proof} 

The fact that there are 2 inequalities above indicates to us that there might be a hope for the {\em desired    inequality } below to be true  for  a finitely supported $X$ with support in $C_k$ and a finitely supported $Y $ with support in $C_l$ we hope that 
\begin{align*}
& \text{ desired inequality }\\ 
& \|M_{\chi_m}*\big(XY\big) \|  \leq  N\|\sum_{a\in C_k}L_aX_aX_a^*L_a^*\|^{\frac{1}{2}}\|\sum_{b\in C_l}Y_b^*Y_b \|^{\frac{1}{2}}
\end{align*}

We may then continue and consider general elemnts $X$ with finite support. 

\begin{theorem} \label{ThmJo}
Let $G$ be a discrete group with a length function  such that $G$ satisfies the condition (J). There exists an $M>0$ such that for any action $\a_g$ of $G$ on a C*-algebra $\ca$ and any $X = \sum_{g \in G}  L_gX_g$ of finite support in $C^*_r(\ca \rtimes_\a G):$
$$\|X\| \leq M  \sqrt{\sum_{g \in G} (1+|g|)^4 \|X_g\|^2}.$$  
\end{theorem}

\begin{proof}
We may proceed as in the proof of Lemma 1.5 in \cite{Ha}, so 
\begin{align*}
\|X\| \leq & \sum_{ k=0}^\infty \| \sum_{g \in C_k} X_g \| \\ \leq & N \sum_{k=0}^\infty (k+1)^{-1} \bigg((k+1)^2\big( \sum_{g \in C_k} \|X_g\|^2\big)^{\frac{1}{2}}\bigg)\\ \notag 
\leq & N \frac{\pi}{\sqrt{6}}\bigg(\sum_{k=0}^\infty (k+1)^4 \sum_{g \in C_k} \|X_g\|^2 \bigg)^{\frac{1}{2}} \\  \notag 
\leq & N \frac{\pi}{\sqrt{6}}\sqrt{2} \bigg(\sum_{g \in G}(1+ |g|)^4  \|X_g\|^2 \bigg)^{\frac{1}{2}}
\end{align*}  and  the theorem follows, since for  $g \in C_k$ we have $k-1 < |g| \leq k.$  
\end{proof}

Again there is a sharper estimate in the case of a free non abelian group.

\begin{corollary}
If $G$ is a free non abelian group the constant $M$ may be chosen as $M=2.$ 
\end{corollary}  

\section{Applications}

The theory of {\em rapid decay } for group C*-algebras has been applied to various types of approximation properties  for
operator algebras \cite{CI}, \cite{CCJJV} \cite{CH}, \cite{HK},  and many research  articles are based on, or inspired by these works. 

\smallskip  
Jolissaint realized from the beginning \cite{J2} that the rapid decay property makes it possible to base some K-theoretical computations on a subalgebra of {\em rapidly decreasing  operators,  } and this was then used by Lafforgue  \cite{La} in his fundamental work on   the Baum-Connes conjecture.

\smallskip
The construction of a spectral triple for a reduced group C*-algebra of a discrete group, which occurs in Connes' non commutative geometry, has an obvious candidate if the group has a length function. It seems natural to use the property {\em rapid decay } to get some information on the properties of this spectral triple, and such attempts have appeared in \cite{AC} and \cite{OR}. It is interesting to see that the so-called {\em Haagerup } condition of \cite{OR} is the content of the very basic Lemma 1.3 of \cite{Ha} and of Proposition 3.2.4 of  \cite{Jo}. Here it is contained in the corollary \ref{HagProp}.

\smallskip
It is not our intent to pursue  possible extensions of the results based on {\em rapid decay} from the group algebra case to the crossed product setting, but we have made one easy  observation, which may be applied to a possible extension of some of the  approximation properties. 

In Haagerup's first article \cite{Ha} he shows in Lemma  1.7 that for a function $\f$ on a free non abelain group $G$ the multiplier $M_\f$ is bounded if $\sup\{|\f(g)|(1 + |g|)^2: g \in G\}$ is finite and $\|M_\f\| \leq 2 \sup\{|\f(g)|(1 + |g|)^2: g \in G\}.$ This result makes it possible for him to cut the completely positive multiplier of norm 1 given by $M_{\f_\l} $ with $\f_\l(g) := e^{-\l |g|}$ to the subsets $ B_n,$ and in this way he obtains a bounded approximate multiplier unit consisting of functions with finite support.   We can not obtain such a nice result here  because  Haagerup's estimate is based on the fact that in the group algebra case we have $\|\l(f)\| \geq \|f\|_2,$ and the analogous statement for crossed products is not true. We can get a result which is is similar to Haagerup's Lemma 1.7 for a group action which has operator rapid decay.
\begin{proposition} \label{Mult}
Let $\ca$ be a C*-algebra, G a discrete group with a length function and $\a_g$ an action of the group on $\ca$ such that the reduced crossed product has operator rapid decay  with coefficients $C, s.$ If a complex function $\f$ on the group satisfies $ m:= \sup\{|\f(g)|(2+|g|)^{(s+1)}\, : \, g \in G\} < \infty $ then the multiplier $M_\f$   on $C^*_r(\ca \rtimes_\a G) $ is  bounded and satisfies $\|M_\f\| \leq 2Cm.$ If the action of $\a_g$ has complete operator rapid decay, then $M_\f$ is completely bounded with $\|M_\f\|_{cb } \leq 4Cm.$
\end{proposition} 
\begin{proof}
Suppose $\f$ is given with $m$ finite then for any $X = \sum L_gX_g $ with finite support
 \begin{align}
\|M_\f*X\|^2 & \leq C^2\| \sum_{g \in G}(1+|g|)^{2s} |\f(g)|^2 ( L_gX_gX_g^*L_g^* + X_g^*X_g) \|
\\ \notag & \leq C^2\| \sum_{g \in G}(1+|g|)^{-2}m^2  ( L_gX_gX_g^*L_g^* + X_g^*X_g) \|
\\\notag & \leq  4 C^2m^2 \sum_{k =0}^\infty  (1+k)^{-2}\| \sum_{g \in C_k}  L_gX_gX_g^*L_g^* + X_g^*X_g\| \\\notag & \leq  8 C^2m^2 \frac{\pi^2}{6} \|X\|^2 \\ \notag &\leq 16C^2 m^2\|X\|^2. 
\end{align}

If $G$ possesses complete operator rapid decay, the action $\a_g$ on $\ca$ may be lifted to actions on $M_n(\ca),$ which all have operator rapid decay with coefficients $C,s$ and the result follows. 
\end{proof}

\end{document}